\documentclass{amsart}
\usepackage{amssymb}
\usepackage{amsfonts}

\newtheorem{theorem}{Theorem}
\theoremstyle{plain}

\newtheorem{corollary}{Corollary}

\newtheorem{definition}{Definition}

\newtheorem{lemma}{Lemma}

\newtheorem{proposition}{Proposition}
\newtheorem{remark}{Remark}

\numberwithin{equation}{section}
\input{tcilatex}

\begin{document}
\title[Simpson's Type Inequalities]{On Some Inequalities of Simpson's Type
via $h-$Convex Functions}
\author{MEVL\"{U}T TUN\c{C}$^{\blacksquare }$}
\address{$^{\blacksquare }$Kilis 7 Aral\i k University, Faculty of Science
and Arts, Department of Mathematics, 79000, Kilis, Turkey}
\email{mevluttunc@kilis.edu.tr}
\author{\c{C}ET\.{I}N YILDIZ$^{\blacktriangledown }$}
\address{$^{\blacktriangledown }$Atat\"{u}rk University, K.K. Education
Faculty, Department of Mathematics, 25240, Erzurum, TURKEY}
\email{yildizc@atauni.edu.tr}
\author{ALPER EK\.{I}NC\.{I}$^{\spadesuit }$}
\address{$^{\spadesuit }$A\u{g}r\i\ \.{I}brahim \c{C}e\c{c}en University,
Faculty of Science and Letters, Department of Mathematics, 04100, A\u{g}r\i
, TURKEY}
\email{alperekinci@hotmail.com}
\thanks{$^{\blacksquare }$Corresponding Author}

\begin{abstract}
In this paper, we prove some new inequalities of Simpson's type for
functions whose derivatives of absolute values are $h-$convex and $h-$%
concave functions. Some new estimations are obtained. Also we give some
sophisticated results for some different kinds of convex functions.
\end{abstract}

\maketitle

\section{INTRODUCTION}

The following inequality is well known in the literature as Simpson's
inequality;

\begin{equation}
\frac{1}{b-a}\int_{a}^{b}f\left( x\right) dx-\frac{1}{3}\left[ \frac{f\left(
a\right) +f\left( b\right) }{2}+2f\left( \frac{a+b}{2}\right) \right] \leq 
\frac{1}{2880}\left\Vert f^{\left( 4\right) }\right\Vert _{\infty }\left(
b-a\right) ^{4},  \label{1.1}
\end{equation}

where the mapping $f:\left[ a,b\right] \rightarrow 
\mathbb{R}
$ is assumed to be four times continuously differentiable on the interval
and $f^{\left( 4\right) }$ to be bounded on $\left( a,b\right) $ , that is,%
\begin{equation*}
\left\Vert f^{\left( 4\right) }\right\Vert _{\infty }=\sup_{t\in \left(
a,b\right) }\left\vert f^{\left( 4\right) }\left( t\right) \right\vert
<\infty .
\end{equation*}

For some results which generalize, improve and extend the inequality (\ref%
{1.1}) see the papers \cite{SOZ2}-\cite{AL1}.

\begin{definition}
\bigskip \textit{\cite{GO} We say that }$f:I\rightarrow 
\mathbb{R}
$\textit{\ is Godunova-Levin function or that }$f$\textit{\ belongs to the
class }$Q\left( I\right) $\textit{\ if }$f$\textit{\ is non-negative and for
all }$x,y\in I$\textit{\ and }$t\in \left( 0,1\right) $\textit{\ we have \ \
\ \ \ \ \ \ \ \ \ \ \ }%
\begin{equation}
f\left( tx+\left( 1-t\right) y\right) \leq \frac{f\left( x\right) }{t}+\frac{%
f\left( y\right) }{1-t}.  \label{103}
\end{equation}
\end{definition}

The class $Q(I)$ was firstly described in \cite{GO} by Godunova-Levin. Among
others, it is noted that nonnegative monotone and nonnegative convex
functions belong to this class of functions.

\begin{definition}
\cite{SS1}\textit{\ We say that }$f:I\subseteq 
\mathbb{R}
\rightarrow 
\mathbb{R}
$\textit{\ is a }$P-$\textit{function or that }$f$\textit{\ belongs to the
class }$P\left( I\right) $\textit{\ if }$f$\textit{\ is nonnegative and for
all }$x,y\in I$\textit{\ and }$t\in \left[ 0,1\right] ,$\textit{\ we have}%
\begin{equation}
f\left( tx+\left( 1-t\right) y\right) \leq f\left( x\right) +f\left(
y\right) .  \label{104}
\end{equation}
\end{definition}

\begin{definition}
\textit{\cite{WW} Let }$s\in \left( 0,1\right] $ be a fixed real number$.$%
\textit{\ A function }$f:[0,\infty )\rightarrow \lbrack 0,\infty )$\textit{\
is said to be }$s-$\textit{convex (in the second sense) if \ \ \ \ \ \ \ \ \
\ \ \ }%
\begin{equation}
f\left( tx+\left( 1-t\right) y\right) \leq t^{s}f\left( x\right) +\left(
1-t\right) ^{s}f\left( y\right) ,  \label{105}
\end{equation}%
\textit{for all }$x,y\in \lbrack 0,\infty )$\textit{\ \ and }$t\in \left[ 0,1%
\right] $\textit{. This class of s-convex functions is usually denoted by }$%
K_{s}^{2}$\textit{.}
\end{definition}

In 1978, Breckner introduced $s-$convex functions as a generalization of
convex functions \textit{\cite{WW}. }Also, in that one work Breckner proved
the important fact that the setvalued map is $s-$convex only if the
associated support function is s-convex function \cite{WW2}. Of course, $s-$%
convexity means just convexity when\textit{\ }$s=1.$

\begin{definition}
\bigskip \textit{\cite{SA} Let }$h:J\subseteq 
\mathbb{R}
\rightarrow 
\mathbb{R}
$\textit{\ be a positive function . We say that }$f:I\subseteq 
\mathbb{R}
\rightarrow 
\mathbb{R}
$\textit{\ is }$h-$\textit{convex function, or that }$f$\textit{\ belongs to
the class }$SX\left( h,I\right) $\textit{, if }$f$\textit{\ is nonnegative
and for all }$x,y\in I$\textit{\ and }$t\in (0,1)$\textit{\ we have \ \ \ \
\ \ \ \ \ \ \ \ }%
\begin{equation}
f\left( tx+\left( 1-t\right) y\right) \leq h\left( t\right) f\left( x\right)
+h\left( 1-t\right) f\left( y\right) .  \label{106}
\end{equation}
\end{definition}

\bigskip If inequality (\ref{106}) is reversed, then $f$ is said to be $h-$%
concave, i.e. $f\in SV\left( h,I\right) $. Obviously, if $h\left( t\right)
=t $, then all nonnegative convex functions belong to $SX\left( h,I\right) $%
\ and all nonnegative concave functions belong to $SV\left( h,I\right) $; if 
$h\left( t\right) =\frac{1}{t}$, then $SX\left( h,I\right) =Q\left( I\right) 
$; if $h\left( t\right) =1$, then $SX\left( h,I\right) \supseteq P\left(
I\right) $; and if $h\left( t\right) =t^{s}$, where $s\in \left( 0,1\right) $%
, then $SX\left( h,I\right) \supseteq K_{s}^{2}$.

\begin{remark}
\textit{\cite{SA} }Let $h$ be a non-negative function such that%
\begin{equation*}
h\left( \alpha \right) \geq \alpha
\end{equation*}%
for all $\alpha \in (0,1)$. If $f$ is a non-negative convex function on $I$
, then for $x,y\in I$ , $\alpha \in (0,1)$ we have 
\begin{equation*}
f\left( \alpha x+(1-\alpha )y\right) \leq \alpha f(x)+(1-\alpha )f(y)\leq
h(\alpha )f(x)+h(1-\alpha )f(y).
\end{equation*}%
So, $f\in SX(h,I)$. Similarly, if the function $h$ has the property: $%
h(\alpha )\leq \alpha $ for all $\alpha \in (0,1)$, then any non-negative
concave function $f$ belongs to the class $SV(h,I)$.
\end{remark}

\begin{definition}
\bigskip\ \textit{\cite{SA} }A function $h:J\rightarrow 
\mathbb{R}
$ is said to be a supermultiplicative function if%
\begin{equation}
h\left( xy\right) \geq h\left( x\right) h\left( y\right)  \label{109}
\end{equation}%
for all $x,y\in J.$
\end{definition}

If inequality (\ref{109}) is reversed, then $h$ is said to be a
submultiplicative function. If equality held in (\ref{109}), then $h$ is
said to be a multiplicative function.

In \cite{SOZ2}, Sar\i kaya \textit{et.al} established the following
Simpson-type inequality for convex functions:

\begin{theorem}
Let $f:I\subset 
\mathbb{R}
\rightarrow 
\mathbb{R}
,$ be differentiable mapping on $I^{\circ }$ such that $f^{\prime }\in
L_{1}[a,b],$ where $a,b\in I$ with $a<b.$ If $\left\vert f^{\prime
}\right\vert $ is a convex on $[a,b],$ then the following inequality holds:%
\begin{eqnarray}
&&\left\vert \frac{1}{6}\left[ f\left( a\right) +4f\left( \frac{a+b}{2}%
\right) +f\left( b\right) \right] -\frac{1}{b-a}\int_{a}^{b}f\left( x\right)
dx\right\vert  \label{b} \\
&\leq &\frac{5(b-a)}{72}\left[ \left\vert f^{\prime }(a)\right\vert
+\left\vert f^{\prime }(b)\right\vert \right]  \notag
\end{eqnarray}
\end{theorem}

\bigskip In \cite{SAR}, Sar\i kaya \textit{et.al} established the following
Hadmard-type inequality for $h-$convex functions:

\begin{theorem}
Let $f\in SX\left( h,I\right) $, $a,b\in I,$ with $a<b$ and $f\in
L_{1}\left( \left[ a,b\right] \right) $. Then%
\begin{equation}
\frac{1}{2h\left( \frac{1}{2}\right) }f\left( \frac{a+b}{2}\right) \leq 
\frac{1}{b-a}\int_{a}^{b}f\left( x\right) dx\leq \left[ f\left( a\right)
+f\left( b\right) \right] \int_{0}^{1}h\left( t\right) dt.  \label{102}
\end{equation}
\end{theorem}

\bigskip For recent results and generalizations concerning $h-$convex
functions see \cite{SA}, \cite{SAR}.

The aim of this paper is to establish new inequalities for functions whose
derivatives in absolute value are $h-$convex and $h-$concave functions.

\section{\protect\bigskip Inequalities for $h-$convex and $h-$concave
functions}

To prove our new result we need the following lemma (see \cite{AL1}).

\begin{lemma}
\label{LEM}Let $f:I\subset 
\mathbb{R}
\rightarrow 
\mathbb{R}
$ be an absolutely continuous mapping on $I^{0}$ where $a,b\in I$ with $a<b.$
Then the following equality holds:%
\begin{eqnarray*}
&&\frac{1}{6}\left[ f\left( a\right) +4f\left( \frac{a+b}{2}\right) +f\left(
b\right) \right] -\frac{1}{b-a}\int_{a}^{b}f\left( x\right) dx \\
&=&\left( b-a\right) \int_{0}^{1}k\left( t\right) f^{\prime }\left(
ta+\left( 1-t\right) b\right) dt,
\end{eqnarray*}

where%
\begin{equation*}
k\left( t\right) =\left\{ 
\begin{array}{c}
t-\frac{1}{6},\text{\ \ }t\in \left[ 0,\frac{1}{2}\right) \\ 
t-\frac{5}{6},\text{ \ }t\in \left[ \frac{1}{2},1\right]%
\end{array}%
\right. .
\end{equation*}
\end{lemma}

\begin{theorem}
\label{t1}Let $h:J\subset 
\mathbb{R}
\rightarrow 
\mathbb{R}
\ $be a non-negative function$,\ f:I\subset \left[ 0,\infty \right)
\rightarrow 
\mathbb{R}
$ be a differentiable function on $I^{\circ }$ such that $h^{q},f^{\prime
}\in L\left[ a,b\right] $ where $a,b\in I^{\circ }$ with $a<b$ and $h\left(
\alpha \right) \geq \alpha $. If $\left\vert f^{\prime }\right\vert $ is $h-$%
convex on $I$, then%
\begin{eqnarray}
&&  \label{a} \\
&&\left\vert \frac{1}{b-a}\int_{a}^{b}f\left( x\right) dx-\frac{1}{3}\left[ 
\frac{f\left( a\right) +f\left( b\right) }{2}+2f\left( \frac{a+b}{2}\right) %
\right] \right\vert   \notag \\
&\leq &\frac{\left( b-a\right) }{3}\left( \frac{1+2^{p+1}}{6\left(
p+1\right) }\right) ^{\frac{1}{p}}\times \left\{ \left\vert f^{\prime
}\left( a\right) \right\vert \left[ \left( \int_{0}^{\frac{1}{2}}h^{q}\left(
t\right) dt\right) ^{\frac{1}{q}}+\left( \int_{\frac{1}{2}}^{1}h^{q}\left(
t\right) dt\right) ^{\frac{1}{q}}\right] \right.   \notag \\
&&+\left. \left\vert f^{\prime }\left( b\right) \right\vert \left[ \left(
\int_{0}^{\frac{1}{2}}h^{q}\left( 1-t\right) dt\right) ^{\frac{1}{q}}+\left(
\int_{\frac{1}{2}}^{1}h^{q}\left( 1-t\right) dt\right) ^{\frac{1}{q}}\right]
\right\} .  \notag
\end{eqnarray}%
where $\frac{1}{p}+\frac{1}{q}=1.$
\end{theorem}

\begin{proof}
From Lemma \ref{LEM}, $h-$convexity of $\left\vert f^{\prime }\right\vert $
and properties of absolute value, we have%
\begin{eqnarray*}
&&\left\vert \frac{1}{b-a}\int_{a}^{b}f\left( x\right) dx-\frac{1}{3}\left[ 
\frac{f\left( a\right) +f\left( b\right) }{2}+2f\left( \frac{a+b}{2}\right) %
\right] \right\vert \\
&=&\left( b-a\right) \left\vert \int_{0}^{1}k\left( t\right) f^{\prime
}\left( ta+\left( 1-t\right) b\right) dt\right\vert \\
&\leq &\left( b-a\right) \left( \int_{0}^{\frac{1}{2}}\left\vert t-\frac{1}{6%
}\right\vert \left\vert f^{\prime }\left( ta+\left( 1-t\right) b\right)
\right\vert dt+\int_{\frac{1}{2}}^{1}\left\vert t-\frac{5}{6}\right\vert
\left\vert f^{\prime }\left( ta+\left( 1-t\right) b\right) \right\vert
dt\right) \\
&\leq &\left( b-a\right) \left\{ \int_{0}^{\frac{1}{2}}\left\vert t-\frac{1}{%
6}\right\vert \left( h\left( t\right) \left\vert f^{\prime }\left( a\right)
\right\vert +h\left( 1-t\right) \left\vert f^{\prime }\left( b\right)
\right\vert \right) dt\right. \\
&&\left. +\int_{\frac{1}{2}}^{1}\left\vert t-\frac{5}{6}\right\vert \left(
h\left( t\right) \left\vert f^{\prime }\left( a\right) \right\vert +h\left(
1-t\right) \left\vert f^{\prime }\left( b\right) \right\vert \right)
dt\right\} \\
&=&\left( b-a\right) \left\vert f^{\prime }\left( a\right) \right\vert
\left\{ \int_{0}^{\frac{1}{2}}\left\vert t-\frac{1}{6}\right\vert h\left(
t\right) dt+\int_{\frac{1}{2}}^{1}\left\vert t-\frac{5}{6}\right\vert
h\left( t\right) dt\right\} \\
&&+\left( b-a\right) \left\vert f^{\prime }\left( b\right) \right\vert
\left\{ \int_{0}^{\frac{1}{2}}\left\vert t-\frac{1}{6}\right\vert h\left(
1-t\right) dt+\int_{\frac{1}{2}}^{1}\left\vert t-\frac{5}{6}\right\vert
h\left( 1-t\right) dt\right\} .
\end{eqnarray*}

By H\"{o}lder's inequality, we get%
\begin{eqnarray*}
&&\left\vert \frac{1}{b-a}\int_{a}^{b}f\left( x\right) dx-\frac{1}{3}\left[ 
\frac{f\left( a\right) +f\left( b\right) }{2}+2f\left( \frac{a+b}{2}\right) %
\right] \right\vert \\
&\leq &\left( b-a\right) \left\vert f^{\prime }\left( a\right) \right\vert
\left\{ \left( \int_{0}^{\frac{1}{2}}\left\vert t-\frac{1}{6}\right\vert
^{p}dt\right) ^{\frac{1}{p}}\left( \int_{0}^{\frac{1}{2}}h^{q}\left(
t\right) dt\right) ^{\frac{1}{q}}\right. \\
&&\text{ \ \ \ \ \ \ \ \ \ \ \ \ \ \ \ \ \ \ }\left. +\left( \int_{\frac{1}{2%
}}^{1}\left\vert t-\frac{5}{6}\right\vert ^{p}dt\right) ^{\frac{1}{p}}\left(
\int_{\frac{1}{2}}^{1}h^{q}\left( t\right) dt\right) ^{\frac{1}{q}}\right\}
\\
&&+\left( b-a\right) \left\vert f^{\prime }\left( b\right) \right\vert
\left\{ \left( \int_{0}^{\frac{1}{2}}\left\vert t-\frac{1}{6}\right\vert
^{p}dt\right) ^{\frac{1}{p}}\left( \int_{0}^{\frac{1}{2}}h^{q}\left(
1-t\right) dt\right) ^{\frac{1}{q}}\right. \\
&&\text{ \ \ \ \ \ \ \ \ \ \ \ \ \ \ \ \ \ \ \ \ \ \ }\left. +\left( \int_{%
\frac{1}{2}}^{1}\left\vert t-\frac{5}{6}\right\vert ^{p}dt\right) ^{\frac{1}{%
p}}\left( \int_{\frac{1}{2}}^{1}h^{q}\left( 1-t\right) dt\right) ^{\frac{1}{q%
}}\right\}
\end{eqnarray*}%
Since%
\begin{equation*}
\int_{0}^{\frac{1}{2}}\left\vert t-\frac{1}{6}\right\vert ^{p}dt=\int_{0}^{%
\frac{1}{6}}\left( \frac{1}{6}-t\right) ^{p}dt+\int_{\frac{1}{6}}^{\frac{1}{2%
}}\left( t-\frac{1}{6}\right) ^{p}dt=\frac{1}{p+1}\left( \frac{1}{6^{p+1}}+%
\frac{1}{3^{p+1}}\right)
\end{equation*}%
and%
\begin{equation*}
\int_{\frac{1}{2}}^{1}\left\vert t-\frac{5}{6}\right\vert ^{p}dt=\int_{\frac{%
1}{2}}^{\frac{5}{6}}\left( \frac{5}{6}-t\right) ^{p}dt+\int_{\frac{5}{6}%
}^{1}\left( t-\frac{5}{6}\right) ^{p}dt=\frac{1}{p+1}\left( \frac{1}{6^{p+1}}%
+\frac{1}{3^{p+1}}\right) ,
\end{equation*}%
we obtain%
\begin{eqnarray*}
&&\left\vert \frac{1}{b-a}\int_{a}^{b}f\left( x\right) dx-\frac{1}{3}\left[ 
\frac{f\left( a\right) +f\left( b\right) }{2}+2f\left( \frac{a+b}{2}\right) %
\right] \right\vert \\
&\leq &\frac{\left( b-a\right) }{3}\left( \frac{1+2^{p+1}}{6\left(
p+1\right) }\right) ^{\frac{1}{p}}\times \left\{ \left\vert f^{\prime
}\left( a\right) \right\vert \left[ \left( \int_{0}^{\frac{1}{2}}h^{q}\left(
t\right) dt\right) ^{\frac{1}{q}}+\left( \int_{\frac{1}{2}}^{1}h^{q}\left(
t\right) dt\right) ^{\frac{1}{q}}\right] \right. \\
&&+\left. \left\vert f^{\prime }\left( b\right) \right\vert \left[ \left(
\int_{0}^{\frac{1}{2}}h^{q}\left( 1-t\right) dt\right) ^{\frac{1}{q}}+\left(
\int_{\frac{1}{2}}^{1}h^{q}\left( 1-t\right) dt\right) ^{\frac{1}{q}}\right]
\right\} .
\end{eqnarray*}%
which completes the proof.
\end{proof}

\begin{corollary}
In Theorem \ref{t1}, if we choose $p=q=2,$ we obtain 
\begin{eqnarray*}
&&\left\vert \frac{1}{b-a}\int_{a}^{b}f\left( x\right) dx-\frac{1}{3}\left[ 
\frac{f\left( a\right) +f\left( b\right) }{2}+2f\left( \frac{a+b}{2}\right) %
\right] \right\vert \\
&\leq &\frac{b-a}{3\sqrt{2}}\left\{ \left\vert f^{\prime }\left( a\right)
\right\vert \left[ \left( \int_{0}^{\frac{1}{2}}h\left( t^{2}\right)
dt\right) ^{\frac{1}{2}}+\left( \int_{\frac{1}{2}}^{1}h\left( t^{2}\right)
dt\right) ^{\frac{1}{2}}\right] \right. \\
&&+\left. \left\vert f^{\prime }\left( b\right) \right\vert \left[ \left(
\int_{0}^{\frac{1}{2}}h\left( \left( 1-t\right) ^{2}\right) dt\right) ^{%
\frac{1}{2}}+\left( \int_{\frac{1}{2}}^{1}h\left( \left( 1-t\right)
^{2}\right) dt\right) ^{\frac{1}{2}}\right] \right\}
\end{eqnarray*}%
where $h$ is supermultiplicative.
\end{corollary}

\begin{corollary}
In Theorem \ref{t1}, if $f\left( a\right) =f\left( \frac{a+b}{2}\right)
=f\left( b\right) $ and $h(t)=1,$ then we have%
\begin{eqnarray*}
&&\left\vert \frac{1}{b-a}\int_{a}^{b}f\left( x\right) dx-f\left( \frac{a+b}{%
2}\right) \right\vert \\
&\leq &\frac{2\left( b-a\right) }{3}\left( \frac{1+2^{p+1}}{6\left(
p+1\right) }\right) ^{\frac{1}{p}}\left( \frac{1}{2}\right) ^{\frac{1}{q}%
}\left\{ \left\vert f^{\prime }\left( a\right) \right\vert +\left\vert
f^{\prime }\left( b\right) \right\vert \right\} \\
&=&\frac{b-a}{3}\left( \frac{1+2^{p+1}}{3\left( p+1\right) }\right) ^{\frac{1%
}{p}}\left\{ \left\vert f^{\prime }\left( a\right) \right\vert +\left\vert
f^{\prime }\left( b\right) \right\vert \right\}
\end{eqnarray*}
\end{corollary}

\begin{corollary}
\label{aa}In Theorem \ref{t1}, if we choose $h\left( t\right) =t$, we obtain%
\begin{eqnarray*}
&&\left\vert \frac{1}{b-a}\int_{a}^{b}f\left( x\right) dx-\frac{1}{3}\left[ 
\frac{f\left( a\right) +f\left( b\right) }{2}+2f\left( \frac{a+b}{2}\right) %
\right] \right\vert \\
&\leq &\frac{\left( b-a\right) }{3}\left( \frac{1+2^{p+1}}{6\left(
p+1\right) }\right) ^{\frac{1}{p}}\left\{ \left\vert f^{\prime }\left(
a\right) \right\vert \left[ \left( \frac{\left( \frac{1}{2}\right) ^{q+1}}{%
q+1}\right) ^{\frac{1}{q}}+\left( \frac{1}{2q+2}\left( 2-\left( \frac{1}{2}%
\right) ^{q}\right) \right) ^{\frac{1}{q}}\right] \right. \\
&&+\left. \left\vert f^{\prime }\left( b\right) \right\vert \left[ \left( 
\frac{1}{2q+2}\left( 2-\left( \frac{1}{2}\right) ^{q}\right) \right) ^{\frac{%
1}{q}}+\left( \frac{\left( \frac{1}{2}\right) ^{q+1}}{q+1}\right) ^{\frac{1}{%
q}}\right] \right\} \\
&=&\frac{\left( b-a\right) }{6}\left( \frac{1+2^{p+1}}{3\left( p+1\right) }%
\right) ^{\frac{1}{p}}\left( \frac{1}{q+1}\right) ^{\frac{1}{q}}\left[ \frac{%
\left\vert f^{\prime }\left( a\right) \right\vert }{2}+\left\vert f^{\prime
}\left( b\right) \right\vert \left( 2-\left( \frac{1}{2}\right) ^{q}\right)
^{\frac{1}{q}}\right] .
\end{eqnarray*}
\end{corollary}

\begin{theorem}
\label{t2}Let $h:J\subset 
\mathbb{R}
\rightarrow 
\mathbb{R}
\ $be a nonnegative supermultiplicative functions$,\ f:I\subset \left[
0,\infty \right) \rightarrow 
\mathbb{R}
$ be a differentiable function on $I^{\circ }$ such that $f^{\prime }\in L%
\left[ a,b\right] $ where $a,b\in I^{\circ }$ with $a<b$ and $h\left( \alpha
\right) \geq \alpha $. If $\left\vert f^{\prime }\right\vert $ is $h-$convex
on $I$, then%
\begin{eqnarray}
&&\left\vert \frac{1}{b-a}\int_{a}^{b}f\left( x\right) dx-\frac{1}{3}\left[ 
\frac{f\left( a\right) +f\left( b\right) }{2}+2f\left( \frac{a+b}{2}\right) %
\right] \right\vert   \notag \\
&\leq &\left( b-a\right) \left\{ \left\vert f^{\prime }\left( a\right)
\right\vert \left[ A\right] +\left\vert f^{\prime }\left( b\right)
\right\vert \left[ B\right] \right\}   \label{22}
\end{eqnarray}%
where%
\begin{eqnarray*}
A &=&\int_{0}^{\frac{1}{6}}h\left( t\left[ \frac{1}{6}-t\right] \right)
dt+\int_{\frac{1}{6}}^{\frac{1}{2}}h\left( t\left[ t-\frac{1}{6}\right]
\right) dt+\int_{\frac{1}{2}}^{\frac{5}{6}}h\left( t\left[ \frac{5}{6}-t%
\right] \right) dt+\int_{\frac{5}{6}}^{1}h\left( t\left[ t-\frac{5}{6}\right]
\right) dt \\
B &=&\int_{0}^{\frac{1}{6}}h\left[ \left( 1-t\right) \left( \frac{1}{6}%
-t\right) \right] dt+\int_{\frac{1}{6}}^{\frac{1}{2}}h\left[ \left(
1-t\right) \left( t-\frac{1}{6}\right) \right] dt \\
&&+\int_{\frac{1}{2}}^{\frac{5}{6}}h\left[ \left( 1-t\right) \left( \frac{5}{%
6}-t\right) \right] dt+\int_{\frac{5}{6}}^{1}h\left[ \left( 1-t\right)
\left( t-\frac{5}{6}\right) \right] dt.
\end{eqnarray*}
\end{theorem}

\begin{proof}
From Lemma \ref{LEM}, $h-$convexity of $\left\vert f^{\prime }\right\vert $
and properties of absolute value, we have%
\begin{eqnarray*}
&&\left\vert \frac{1}{b-a}\int_{a}^{b}f\left( x\right) dx-\frac{1}{3}\left[ 
\frac{f\left( a\right) +f\left( b\right) }{2}+2f\left( \frac{a+b}{2}\right) %
\right] \right\vert \\
&=&\left( b-a\right) \left\vert \int_{0}^{1}k\left( t\right) f^{\prime
}\left( ta+\left( 1-t\right) b\right) dt\right\vert \\
&\leq &\left( b-a\right) \left\{ \int_{0}^{\frac{1}{2}}\left\vert t-\frac{1}{%
6}\right\vert \left\vert f^{\prime }\left( ta+\left( 1-t\right) b\right)
\right\vert dt+\int_{\frac{1}{2}}^{1}\left\vert t-\frac{5}{6}\right\vert
\left\vert f^{\prime }\left( ta+\left( 1-t\right) b\right) \right\vert
dt\right\} \\
&\leq &\left( b-a\right) \left\{ \int_{0}^{\frac{1}{2}}\left\vert t-\frac{1}{%
6}\right\vert \left( h\left( t\right) \left\vert f^{\prime }\left( a\right)
\right\vert +h\left( 1-t\right) \left\vert f^{\prime }\left( b\right)
\right\vert \right) dt\right. \\
&&\text{ \ \ \ \ \ \ \ \ \ \ }\left. +\int_{\frac{1}{2}}^{1}\left\vert t-%
\frac{5}{6}\right\vert \left( h\left( t\right) \left\vert f^{\prime }\left(
a\right) \right\vert +h\left( 1-t\right) \left\vert f^{\prime }\left(
b\right) \right\vert \right) dt\right\} \\
&=&\left( b-a\right) \left\{ \int_{0}^{\frac{1}{6}}\left( \frac{1}{6}%
-t\right) \left( h\left( t\right) \left\vert f^{\prime }\left( a\right)
\right\vert +h\left( 1-t\right) \left\vert f^{\prime }\left( b\right)
\right\vert \right) dt\right. \\
&&\text{ \ \ \ \ \ \ \ \ \ \ }+\int_{\frac{1}{6}}^{\frac{1}{2}}\left( t-%
\frac{1}{6}\right) \left( h\left( t\right) \left\vert f^{\prime }\left(
a\right) \right\vert +h\left( 1-t\right) \left\vert f^{\prime }\left(
b\right) \right\vert \right) dt \\
&&\text{ \ \ \ \ \ \ \ \ \ \ }+\int_{\frac{1}{2}}^{\frac{5}{6}}\left( \frac{5%
}{6}-t\right) \left( h\left( t\right) \left\vert f^{\prime }\left( a\right)
\right\vert +h\left( 1-t\right) \left\vert f^{\prime }\left( b\right)
\right\vert \right) dt \\
&&\text{ \ \ \ \ \ \ \ \ \ \ }\left. +\int_{\frac{5}{6}}^{1}\left( t-\frac{5%
}{6}\right) \left( h\left( t\right) \left\vert f^{\prime }\left( a\right)
\right\vert +h\left( 1-t\right) \left\vert f^{\prime }\left( b\right)
\right\vert \right) dt\right\} .
\end{eqnarray*}%
By properties of function $h$, we can write%
\begin{eqnarray*}
&&\left\vert \frac{1}{b-a}\int_{a}^{b}f\left( x\right) dx-\frac{1}{3}\left[ 
\frac{f\left( a\right) +f\left( b\right) }{2}+2f\left( \frac{a+b}{2}\right) %
\right] \right\vert \\
&\leq &\left( b-a\right) \left\{ \int_{0}^{\frac{1}{6}}\left[ h\left(
t\left( \frac{1}{6}-t\right) \right) \left\vert f^{\prime }\left( a\right)
\right\vert +h\left( \left( 1-t\right) \left( \frac{1}{6}-t\right) \right)
\left\vert f^{\prime }\left( b\right) \right\vert \right] dt\right. \\
&&\text{ \ \ \ \ \ \ \ \ }+\int_{\frac{1}{6}}^{\frac{1}{2}}\left[ h\left(
t\left( t-\frac{1}{6}\right) \right) \left\vert f^{\prime }\left( a\right)
\right\vert +h\left( \left( 1-t\right) \left( t-\frac{1}{6}\right) \right)
\left\vert f^{\prime }\left( b\right) \right\vert \right] dt \\
&&\text{ \ \ \ \ \ \ \ \ }+\int_{\frac{1}{2}}^{\frac{5}{6}}\left[ h\left(
t\left( \frac{5}{6}-t\right) \right) \left\vert f^{\prime }\left( a\right)
\right\vert +h\left( \left( 1-t\right) \left( \frac{5}{6}-t\right) \right)
\left\vert f^{\prime }\left( b\right) \right\vert \right] dt \\
&&\text{ \ \ \ \ \ \ \ \ }\left. +\int_{\frac{5}{6}}^{1}\left[ h\left(
t\left( t-\frac{5}{6}\right) \right) \left\vert f^{\prime }\left( a\right)
\right\vert +h\left( \left( 1-t\right) \left( t-\frac{5}{6}\right) \right)
\left\vert f^{\prime }\left( b\right) \right\vert \right] dt\right\}
\end{eqnarray*}%
\begin{eqnarray*}
&=&\left( b-a\right) \left\vert f^{\prime }\left( a\right) \right\vert
\left\{ \int_{0}^{\frac{1}{6}}h\left( t\left[ \frac{1}{6}-t\right] \right)
dt+\int_{\frac{1}{6}}^{\frac{1}{2}}h\left( t\left[ t-\frac{1}{6}\right]
\right) dt\right. \\
&&\text{ \ \ \ \ \ \ \ \ \ \ \ \ \ \ \ \ \ \ }\left. +\int_{\frac{1}{2}}^{%
\frac{5}{6}}h\left( t\left[ \frac{5}{6}-t\right] \right) dt+\int_{\frac{5}{6}%
}^{1}h\left( t\left[ t-\frac{5}{6}\right] \right) dt\right\} \\
&&+\left( b-a\right) \left\vert f^{\prime }\left( b\right) \right\vert
\left\{ \int_{0}^{\frac{1}{6}}h\left( \left( 1-t\right) \left( \frac{1}{6}%
-t\right) \right) dt+\int_{\frac{1}{6}}^{\frac{1}{2}}h\left( \left(
1-t\right) \left( t-\frac{1}{6}\right) \right) dt\right. \\
&&\text{ \ \ \ \ \ \ \ \ \ \ \ \ \ \ \ \ \ \ \ }\left. +\int_{\frac{1}{2}}^{%
\frac{5}{6}}h\left[ \left( 1-t\right) \left( \frac{5}{6}-t\right) \right]
dt+\int_{\frac{5}{6}}^{1}h\left[ \left( 1-t\right) \left( t-\frac{5}{6}%
\right) \right] dt\right\}
\end{eqnarray*}%
which completes the proof.
\end{proof}

\begin{corollary}
\label{bb}In Theorem \ref{t2}, if $f\left( a\right) =f\left( \frac{a+b}{2}%
\right) =f\left( b\right) $ and $h(t)=1,$ then we have%
\begin{eqnarray*}
&&\left\vert \frac{1}{b-a}\int_{a}^{b}f\left( x\right) dx-f\left( \frac{a+b}{%
2}\right) \right\vert \\
&\leq &\left( b-a\right) \left\{ \left\vert f^{\prime }\left( a\right)
\right\vert +\left\vert f^{\prime }\left( b\right) \right\vert \right\} .
\end{eqnarray*}
\end{corollary}

\begin{remark}
In Theorem \ref{t2}, if we choose $h(t)=t,$ then we obtain the inequality (%
\ref{b}).
\end{remark}

\begin{theorem}
\label{to3}Let $h:J\subset 
\mathbb{R}
\rightarrow 
\mathbb{R}
\ $be non-negative functions$,\ f:I\subset \left[ 0,\infty \right)
\rightarrow 
\mathbb{R}
$ be a differentiable function on $I^{\circ }$ such that $f^{\prime }\in L%
\left[ a,b\right] $ where $a,b\in I^{\circ }$ with $a<b$. If $\left\vert
f^{\prime }\right\vert $ is $h-$concave on $I$, then%
\begin{eqnarray*}
&&\left\vert \frac{1}{b-a}\int_{a}^{b}f\left( x\right) dx-\frac{1}{3}\left[ 
\frac{f\left( a\right) +f\left( b\right) }{2}+2f\left( \frac{a+b}{2}\right) %
\right] \right\vert  \\
&\leq &\frac{b-a}{12}\left( \frac{2+2^{p+2}}{p+1}\right) ^{\frac{1}{p}}\left[
\frac{1}{h\left( \frac{1}{2}\right) }\right] ^{\frac{1}{q}}\left\vert
f^{\prime }\left( \frac{a+b}{2}\right) \right\vert .
\end{eqnarray*}
\end{theorem}

\begin{proof}
From Lemma \ref{LEM} and using the H\"{o}lder inequality, we have%
\begin{eqnarray*}
&&\left\vert \frac{1}{b-a}\int_{a}^{b}f\left( x\right) dx-\frac{1}{3}\left[ 
\frac{f\left( a\right) +f\left( b\right) }{2}+2f\left( \frac{a+b}{2}\right) %
\right] \right\vert \\
&\leq &\left( b-a\right) \left\vert \int_{0}^{1}k\left( t\right) f^{\prime
}\left( ta+\left( 1-t\right) b\right) dt\right\vert \\
&\leq &\left( b-a\right) \left( \int_{0}^{1}\left\vert k\left( t\right)
\right\vert ^{p}dt\right) ^{\frac{1}{p}}\left( \int_{0}^{1}\left\vert
f^{\prime }\left( ta+\left( 1-t\right) b\right) \right\vert ^{q}dt\right) ^{%
\frac{1}{q}}
\end{eqnarray*}%
Since $\left\vert f^{\prime }\right\vert $ is $h$-concave on $I,$ by
inequalities (\ref{102}) we have%
\begin{equation*}
\int_{0}^{1}\left\vert f^{\prime }\left( ta+\left( 1-t\right) b\right)
\right\vert ^{q}dt\leq \frac{1}{2h\left( \frac{1}{2}\right) }\left\vert
f^{\prime }\left( \frac{a+b}{2}\right) \right\vert ^{q}.
\end{equation*}%
Therefore, we get%
\begin{eqnarray*}
&&\left\vert \frac{1}{b-a}\int_{a}^{b}f\left( x\right) dx-\frac{1}{3}\left[ 
\frac{f\left( a\right) +f\left( b\right) }{2}+2f\left( \frac{a+b}{2}\right) %
\right] \right\vert \\
&\leq &(b-a)\left\{ \int_{0}^{\frac{1}{2}}\left\vert t-\frac{1}{6}%
\right\vert ^{p}dt+\int_{\frac{1}{2}}^{1}\left\vert t-\frac{5}{6}\right\vert
^{p}dt\right\} ^{\frac{1}{p}}\left[ \frac{1}{2h\left( \frac{1}{2}\right) }%
\left\vert f^{\prime }\left( \frac{a+b}{2}\right) \right\vert ^{q}\right] ^{%
\frac{1}{q}}.
\end{eqnarray*}%
Since%
\begin{equation*}
\int_{0}^{\frac{1}{2}}\left\vert t-\frac{1}{6}\right\vert ^{p}dt=\int_{0}^{%
\frac{1}{6}}\left( \frac{1}{6}-t\right) ^{p}dt+\int_{\frac{1}{6}}^{\frac{1}{2%
}}\left( t-\frac{1}{6}\right) ^{p}dt=\frac{1}{p+1}\left( \frac{1}{6^{p+1}}+%
\frac{1}{3^{p+1}}\right)
\end{equation*}%
and%
\begin{equation*}
\int_{\frac{1}{2}}^{1}\left\vert t-\frac{5}{6}\right\vert ^{p}dt=\int_{\frac{%
1}{2}}^{\frac{5}{6}}\left( \frac{5}{6}-t\right) ^{p}dt+\int_{\frac{5}{6}%
}^{1}\left( t-\frac{5}{6}\right) ^{p}dt=\frac{1}{p+1}\left( \frac{1}{6^{p+1}}%
+\frac{1}{3^{p+1}}\right) ,
\end{equation*}%
we obtain%
\begin{eqnarray*}
&&\left\vert \frac{1}{b-a}\int_{a}^{b}f\left( x\right) dx-\frac{1}{3}\left[ 
\frac{f\left( a\right) +f\left( b\right) }{2}+2f\left( \frac{a+b}{2}\right) %
\right] \right\vert \\
&\leq &\frac{b-a}{12}\left( \frac{2+2^{p+2}}{p+1}\right) ^{\frac{1}{p}}\left[
\frac{1}{h\left( \frac{1}{2}\right) }\right] ^{\frac{1}{q}}\left\vert
f^{\prime }\left( \frac{a+b}{2}\right) \right\vert .
\end{eqnarray*}%
This completes the proof.
\end{proof}

\begin{corollary}
In Theorem \ref{to3}, if we choose $h(t)=t,$ then we obtain%
\begin{eqnarray*}
&&\left\vert \frac{1}{b-a}\int_{a}^{b}f\left( x\right) dx-\frac{1}{3}\left[ 
\frac{f\left( a\right) +f\left( b\right) }{2}+2f\left( \frac{a+b}{2}\right) %
\right] \right\vert \\
&\leq &\frac{b-a}{6}\left( \frac{1+2^{p+1}}{p+1}\right) ^{\frac{1}{p}%
}\left\vert f^{\prime }\left( \frac{a+b}{2}\right) \right\vert .
\end{eqnarray*}
\end{corollary}

\section{APPLICATIONS TO SPECIAL MEANS}

We now consider the means for arbitrary real numbers $\alpha ,\beta $ $%
(\alpha \neq \beta ).$ We take

\begin{enumerate}
\item $Arithmetic$ $mean:$%
\begin{equation*}
A(\alpha ,\beta )=\frac{\alpha +\beta }{2},\text{ \ }\alpha ,\beta \in 
\mathbb{R}
^{+}.
\end{equation*}

\item $Logarithmic$ $mean$:%
\begin{equation*}
L(\alpha ,\beta )=\frac{\alpha -\beta }{\ln \left\vert \alpha \right\vert
-\ln \left\vert \beta \right\vert },\text{ \ \ }\left\vert \alpha
\right\vert \neq \left\vert \beta \right\vert ,\text{ }\alpha ,\beta \neq 0,%
\text{ }\alpha ,\beta \in 
\mathbb{R}
^{+}.
\end{equation*}

\item $Generalized$ $log-mean$:%
\begin{equation*}
L_{n}(\alpha ,\beta )=\left[ \frac{\beta ^{n+1}-\alpha ^{n+1}}{(n+1)(\beta
-\alpha )}\right] ^{\frac{1}{n}},\text{ \ \ \ }n\in 
\mathbb{Z}
\backslash \{-1,0\},\text{ }\alpha ,\beta \in 
\mathbb{R}
^{+}.
\end{equation*}
\end{enumerate}

Now using the results of Section 2, we give some applications for special
means of real numbers.

\begin{proposition}
Let $a,b\in 
\mathbb{R}
^{+}$, $0<a<b$ and $n\in 
\mathbb{N}
,n>1.$ Then, we have%
\begin{eqnarray*}
&&\left\vert L_{n}^{n}(a,b)-\frac{1}{3}\left[ A(a^{n},b^{n})-A^{n}(a,b)%
\right] \right\vert \\
&\leq &n\frac{\left( b-a\right) }{6}\left( \frac{1+2^{p+1}}{3\left(
p+1\right) }\right) ^{\frac{1}{p}}\left( \frac{1}{q+1}\right) ^{\frac{1}{q}}%
\left[ \frac{a^{n-1}}{2}+b^{n-1}\left( 2-\left( \frac{1}{2}\right)
^{q}\right) ^{\frac{1}{q}}\right] .
\end{eqnarray*}
\end{proposition}

\begin{proof}
The assertion follows from Corollary \ref{aa} applied for $f(x)=x^{n},$ $%
x\in 
\mathbb{R}
,$ $n\in 
\mathbb{N}
.$
\end{proof}

\begin{proposition}
Let $a,b\in 
\mathbb{R}
^{+}$, $a<b.$ Then, we have%
\begin{equation*}
\left\vert L_{n}^{n}(a,b)-A^{n}(a,b)\right\vert \leq \left( b-a\right) \left[
\frac{1}{a^{2}}+\frac{1}{b^{2}}\right] .
\end{equation*}
\end{proposition}

\begin{proof}
The assertion follows from Corollary \ref{bb} applied for $f(x)=\frac{1}{x},$
$x\in \lbrack a,b].$
\end{proof}

\end{document}